\documentclass[a4paper,12pt,reqno]{amsart}

\usepackage{amsmath, amsfonts, amssymb, amsthm, mathrsfs, mathtools, mathalfa,pstricks,pstricks-add}
\usepackage{comment}
\usepackage{yfonts}
\usepackage{enumerate}
\usepackage{enumitem}
\usepackage{stmaryrd}
\usepackage{fancyhdr}
\usepackage{bm}
\usepackage[utf8]{inputenc}
\usepackage[pdfencoding=auto, hyperfootnotes=false]{hyperref}
\usepackage{tikz,tikz-cd}
\usepackage[hang,flushmargin]{footmisc}
\usepackage{graphicx}
\usepackage{xr}
\usepackage[dvistyle, colorinlistoftodos]{todonotes}
\usepackage{extarrows}

\def\F{\mathbb{F}}
\def\Ann{\mathrm{Ann}}
\def\Orth{\mathrm{Orth}}
\newtheorem{theorem}{Theorem}[section]
\newtheorem{lemma}[theorem]{Lemma}

\newtheorem{proposition}[theorem]{Proposition}
\newtheorem{prop}[theorem]{Proposition}

\theoremstyle{definition}

\newcommand{\mb}{\mathbb}

\DeclareMathOperator{\Span}{span}

\DeclareMathOperator{\codim}{codim}

\providecommand{\abs}[1]{\lvert#1\rvert}

\begin{document}

\title[A note on the bilinear Bogolyubov theorem]{A note on the bilinear Bogolyubov theorem: transverse and bilinear sets}
\author[P.-Y. Bienvenu]{Pierre-Yves Bienvenu}
\address{Institut Camille-Jordan, Universit\'e Lyon 1, 43 boulevard du 11 novembre 1918
69622 Villeurbanne cedex, France}
\email{pbienvenu@math.univ-lyon1.fr}
\author[D. Gonz\'alez-S\'anchez]{Diego Gonz\'alez-S\'anchez}
\address{Instituto de Ciencias Matem\'aticas (CSIC-UAM-UC3M-UCM) -- Departamento de Matem\'aticas (Universidad Aut\'onoma de Madrid), 28049 Madrid, Spain} 
\email{diego.gonzalezs@uam.es}
\author[\'A. D. Mart\'inez]{\'Angel D. Mart\'inez}
\address{Institute for Advanced Study, Fuld Hall 412, 1 Einstein Drive, Princeton NJ 08540, USA} 
\email{amartinez@ias.edu}

\begin{abstract} 
A set $P\subset \F_p^n\times\F_p^n$
 is called \textit{bilinear} when it is the zero set of a family of linear  
and bilinear forms, and \textit{transverse}  
when it is stable under vertical and horizontal sums. A theorem of the  
first author provides a generalization of Bogolyubov's theorem to the  
bilinear setting. Roughly speaking, it implies that any dense  
transverse set $P\subset \F_p^n\times\F_p^n$ contains a large  
bilinear set. In this paper, we elucidate the extent to which  
a transverse set is forced to be (and not only contain) a bilinear set.
\end{abstract}
\date{}
\maketitle

\section{Introduction}

\noindent 
A simple exercise shows that any nonempty subset $A\subset \mb{F}_p^n$ that is closed under addition is a linear subspace, that is, 
the zero set of a family of linear forms.
Indeed, denoting as usual
\[A\pm A=\{a\pm b: (a,b)\in A^2\},\]
this amounts to the claim that $A+A=A\neq \emptyset$ if and only if $A$ is a subspace (and analogously for 
$A-A$).
Considering a large amount of summands, one will eventually 
get $\Span(A)$, the linear subspace generated by $A$.  This may require an unbounded number of summands as the dimension $n$ or the prime $p$ tends to infinity.

The following classical result states that a bounded number of summands already suffices to produce a rather large subspace of $\Span(A)$.

\begin{theorem}[Bogolyubov]\label{thm:bogol} Let $A\subset \mb{F}_p^n$ be a subset of density $\alpha>0$, that is, 
$\abs{A}=\alpha p^n$. Then $2A-2A$
 contains a vector space of codimension $c(\alpha)=O(\alpha^{-2})$.
\end{theorem}

\noindent Bogolyubov's original paper \cite{B39} deals with 
$\mb{Z}/N\mb{Z}$, but the ideas translate to finite $\mb{F}_p$-vector spaces. 
Note that if $A$ is a vector space, its codimension is $\log_p\alpha^{-1}$. As a consequence, $c(\alpha)\geq \log_p\alpha^{-1}$. Sanders \cite{SandersBogo}
improved the bound in the statement to a nearly optimal $c(\alpha)=O(\log^{4}\alpha^{-1})$.
Recently, bilinear versions of this result by the first author and L\^e \cite{BL17} and, independently, Gowers and Mili\'cevi\'c \cite{GM17} have appeared. Let us now state this bilinear Bogolyubov theorem. We need to introduce a piece of useful notation (cf. \cite{BL17}). 

 Given a set $A\subset \mb{F}_p^n \times \mb{F}_p^n$ we define the vertical sum or difference as
\[
A \stackrel{V}{\pm} A:= \{ (x,y_1\pm y_2): (x,y_1),(x,y_2)\in A \}.
\]
The set $A \stackrel{H}{\pm} A$ is defined analogously but fixing the second coordinate. Then we define $\phi_V$ as the operation
\[
A \mapsto (A \stackrel{V}{+} A) \stackrel{V}{-} (A \stackrel{V}{+} A)
\]
and $\phi_H$ similarly.
The theorem proved in \cite{BL17} is the following.

\begin{theorem}[Bienvenu and L\^e, \cite{BL17}]\label{thm:bilinear-bogol} Let $\delta>0$, then there is $c(\delta)>0$ such that the following holds. For any $A \subset \mb{F}_p^n \times \mb{F}_p^n$ of density $\delta$, there exists $W_1,W_2 \subseteq \mb{F}_p^n$ subspaces of codimension $r_1$ and $r_2$ respectively and bilinear  forms $Q_1,\cdots,Q_{r_3}$ on $W_1\times W_2$ such that
$\phi_H\phi_V\phi_H(A)$ contains
\begin{equation}\label{eq:bilinar-bogol}
\{(x,y)\in W_1\times W_2: Q_1(x,y)=\cdots=Q_{r_3}(x,y)=0 \} 
\end{equation}
where $\max\{r_1,r_2,r_3\}\le c(\delta)$.
\end{theorem}
\noindent The poor bound obtained in \cite{BL17} and \cite{GM17} was improved
very recently by Hosseini and Lovett \cite{HL18} to the nearly optimal
$c(\delta)=O(\log^{O(1)}\delta^{-1})$, at the cost of replacing
$\phi_H\phi_V\phi_H$ by a slightly longer sequence of 
operations.

We call a set $A\subset\mathbb{F}^n_p\times\mathbb{F}^n_p$ \textit{transverse} if it satisfies $A\stackrel{V}{+} A=A\stackrel{H}{+} A=A$. In connection with the result 
above the following natural problem arose: characterise transverse sets. Examples of transverse sets
are what we call \textit{bilinear} sets, that is, zero sets of linear and bilinear forms as in \eqref{eq:bilinar-bogol}.
It is tantalizing to suspect that they are the only possible examples.
Theorem \ref{thm:bilinear-bogol} only shows that any transverse set $A$ of
density $\alpha$
\emph{contains} a bilinear subset defined by $c(\alpha)$ linear
and bilinear forms.

In this paper, we find transverse, non bilinear sets
$A\subset\mathbb{F}^n_p\times\mathbb{F}^n_p$ for any $(p,n)$
except $p=2$ and $n=2$ where it is possible to list all transverse sets and check that they are bilinear. In this direction, we provide an explicit counterexample for $p=3$ and $n=2$ and a non-constructive argument in general.

\begin{prop}
\label{algebraic}
Let $P\subset \F_3^2\times\F_3^2$ be the set of 
$((x_1,x_2),(y_1,y_2))$ satisfying
\begin{equation}\label{def:Pintro}
\left\{\begin{array}{l}
x_1y_1^2+x_2y_2^2=0\\
x_1^2y_1+x_2^2y_2=0
\end{array}\right.
\end{equation}
is transverse but not bilinear.
\end{prop}

Nevertheless, we show that transversity together with an extra 
largeness hypothesis implies bilinearity for small characteristics. For any transverse set $P\subset \F_p^n\times \F_p^n$, let
$P_{x\cdot}=\{y\in \F_p^n:(x,y)\in P\}$ be the vertical \textit{fiber} above $x\in\F_p^n$. Notice that a non-empty fiber is a subspace.

\begin{theorem}
\label{positive}
Let $P\subset \F_p^n\times \F_p^n$ be a transverse set such that $P_{x\cdot}$ contains a hyperplane for any $x$. Then it is bilinear provided that the prime $p=2$ or $3$.
\end{theorem}
\noindent We end the paper providing non constructive counterexamples.
\begin{theorem} \label{existence}
Let $p$ be a prime and $n$ a positive integer.
\begin{enumerate}
\item 
 For any prime $p\geq 5$ and dimension $n\geq 2$, there exists a transverse, non-bilinear set $P\subset \F_p^n\times \F_p^n$ for which 
$P_{x\cdot}$ contains a hyperplane for any $x$.
\item For all but finitely many primes $p$ and dimensions $n$,  we 
can find transverse, non-bilinear sets 
$P\subset \F_p^n\times \F_p^n$ where $P_{x\cdot}$ is a space of dimension 1 for any $x$.
\end{enumerate}
\end{theorem}

The paper is organized as follows. In Section \ref{galg} we study the explicit algebraic counterexample. In Section \ref{sec:positive} we provide a qualitative classification of transverse sets
$P$ for which $P_{x\cdot}$ contains a hyperplane;
 this entails a proof for Theorem \ref{positive} and the basis for the proof Theorem \ref{existence}, which can be finally found in Section \ref{thirdAlt}.

\section{Proof of proposition \ref{algebraic}}\label{galg}

Consider $P\subset \mathbb{F}_3^2\times\mathbb{F}_3^2$ to be the set defined by the system \eqref{def:Pintro}.
We want to show that we cannot have 
\[P=\{(x,y)\in W_1\times W_2: Q_1(x,y)=\cdots=Q_{r_3}(x,y)=0 \}\]
for any subspaces $W_1,W_2$ and any bilinear forms $Q_1,\cdots Q_{r_3}$
 so by contradiction suppose that it is the case.

The set $P$ is easy to describe: indeed, 
if $(x,y)\in P$, then either $x_1y_1=x_2y_2=0$
or $x_1y_1x_2y_2\neq 0$. Let 
\[P_0=\{(x_1,x_2,y_1,y_2)\in \mb{F}_3^2 \times \mb{F}_3^2: x_1y_1=0 \text{ and } x_2y_2=0\}\]
 and \[P_1=\{(x_1,x_2,y_1,y_2)\in \mb{F}_3^2 \times \mb{F}_3^2: x_1+x_2=0 \text{ and } y_1+y_2=0\}\] which is a subset of $P$ and contains the set of points where $x_1y_1x_2y_2\neq 0$  since $z^2\equiv 1\mod 3$ provided $z\not\equiv 0\mod 3$. Therefore $P=P_0 \cup P_1$.\\

Let us check that this set satisfies both conditions $P\stackrel{V}{+}P=P$ and $P\stackrel{H}{+}P=P$. By symmetry it is enough to check that $P\stackrel{H}{+}P=P$. The cases where the points $(x_1,x_2,y_1,y_2), (x'_1,x'_2,y_1,y_2)$ are both in $P_0$ or $P_1$ are easily verified and if one is in $P_0$ and the other in $P_1$ then $(x_1+x'_1)y_1^2+(x_2+x'_2)y_2^2=0$ by the first equation in \eqref{def:Pintro} and 
\[(x_1+x'_1)^2y_1+(x_2+x'_2)^2y_2=2(x_1x_1'y_1+x_2x_2'y_2)=0\]
 using the fact that either $(x_1,x_2,y_1,y_2)$ or $ (x'_1,x'_2,y_1,y_2)$ is in $P_0$.\\

The fact that $P_1 \subset P$ shows that $W_1$, $W_2$ are at least one dimensional but this is not enough. Indeed, suppose they are one dimensional, then $W_1$ and $W_2$ should be precisely $\{(x_1,x_2): x_1+x_2=0\}$ and $\{(y_1,y_2): y_1+y_2=0\}$ but, for example, $(1,0,0,0)\notin W_1\times \F_3^2$ and $(0,0,1,0)\notin \F_3^2\times W_2$ and they belong to $P$. As a consequence $W_1=W_2=\mathbb{F}_3^2$. Let us show that no bilinear form other than the trivial one can vanish on this $P$. Suppose
\[xQy=\left(\begin{array}{cc}
x_1&x_2\end{array}\right)
\left(\begin{array}{cc}
a_{11}&a_{12}\\
a_{21}&a_{22}
\end{array}\right)
\left(\begin{array}{c}
y_1\\
y_2\end{array}\right)=0\]
for all $(x,y)\in P$ or, alternatively,
\[xQy=a_{11}x_1y_1+a_{12}x_1y_2+a_{21}x_2y_1+a_{22}x_2y_2=0.\]
On $P_0\subset P$, this equation boils down to
\[a_{12}x_1y_2+a_{21}x_2y_1=0\]
but now $(0,1,1,0),(1,0,0,1)\in P_0$ imply $a_{12}=a_{21}=0$. On the other hand $(1,2,1,2)\in P_1$ imply $a_{11}+a_{22}=0$. This  implies that if $P$ is a bilinear set then it must be the zero set of $Q=\left(\begin{array}{cc}1&0\\0&-1\end{array}\right)$ (or equivalently, $-Q$). But this is impossible because $(x,y)=(1,1,1,1)\notin P$ and yet $xQy=0$. 
So the only option left is that $P=\mathbb{F}_3^2\times\mathbb{F}_3^2$ and this is not the case either. As an aside, note that $\dim P_{x\cdot}$ is not constant on 
$\F_p^2\setminus\{0\}$, so this example is different from the generic
ones mentioned in Theorem \ref{existence}.

\section{Proof of proposition \ref{positive}}
\label{sec:positive}
In this section, we 
prove Theorem \ref{positive}
Let $V_1$ and $V_2$ be two $\F_p$-vector spaces,
and we slightly generalise the above discussion to transverse sets of $V_1\times V_2$.
Let $P\subset V_1\times V_2$ be a set.
Write $P_{x\cdot}=\{y\in V_2 : (x,y)\in P\}$ and $P_{\cdot y}=\{x\in V_1: (x,y)\in P\}$ for the vertical and horizontal fibers, respectively, borrowing the notation from \cite{GM17}.
We now characterise transversity by some rigidity property of
the map $x\mapsto P_{x\cdot}$.

\begin{lemma}
\label{rigidityPx}
A set $P\subset V_1\times V_2$ is transverse if, and only if,
the map $x\mapsto P_{x\cdot}$ satisfies the following properties.
\begin{enumerate}
\item For any $x$, the set $P_{x\cdot}$ is the empty set or a subspace and $P_{x\cdot}\subset P_{0\cdot}$.
\item For any $x\neq 0$, the set $P_{x\cdot}$
depends only on the class $[x]\in P(V_1)=V_1^*/\F_p^*$
of $x$ in the projective space. 
\item If $[z]$ is on the projective line spanned by $[x]$ and $[y]$, we have $P_{z\cdot}\supset P_{x\cdot}\cap P_{y\cdot}$.
\end{enumerate}
\end{lemma}
\begin{proof}
Let $P\subset V_1\times V_2$ be transverse.
Let $x\in V_1$. Because $P \stackrel{V}{+} P$, we find that $P_{x\cdot}+P_{x\cdot}=P_{x\cdot}$, so $P_{x\cdot}$ is empty or a subspace. Similarly $P_{\cdot y}$ is empty or a subspace. Let $y\in P_{x\cdot}$.
Then $x\in P_{\cdot y}$ which implies $0\in P_{\cdot y}$, 
hence $y\in P_{0\cdot}$, which proves the first point. Further, 
$(\lambda x,y)\in P$ for any $\lambda \neq 0$ as well, thus $y\in P_{\lambda x \cdot}$; this shows the second point.
To prove the third point, suppose 
without loss of generality that $z=x+\lambda y$ for some $\lambda \in \F_p$. Let $w\in  P_{x\cdot}\cap P_{y\cdot}$.
Thus both $x$ and $y$ belong to the subspace $P_{\cdot w}$, so that
$z\in P_{\cdot w}$ too, which means that $w\in P_{z\cdot}$, concluding the proof.

We now prove the converse. Let a set $P\subset V_1\times V_2$ satisfy the three properties.
The first point means that $P \stackrel{V}{+} P=P$.
The horizontal stability follows from the second and third points.
\end{proof}

We will need another lemma.
Recall the notation $\mathbb{P}(V)=V^*/\F^*$ for the projective space of an $\F$-vector space $V$.
We will often omit the distinction between $x\in V$ and
its class $[x]\in \mathbb{P}(V)$. It will be convenient to use the language of projective geometry, of which we assume some basic facts, such as
the fact that any two (projective) lines of a (projective) plane intersect.
\begin{lemma}
\label{collineation}
Suppose that $\xi : \mathbb{P}(V_1)\rightarrow \mathbb{P}(V_2)$ has the property 
that for any $x,y,z$ in $V_1$ such that $z\in\Span(x,y)$, we have
$\xi(z)\in\Span(\xi(x),\xi(y))$. Then $\xi$ is either constant
or injective.
\end{lemma}
\begin{proof}
First we deal with the case where $\mathbb{P}(V_1)$ is a projective line (i.e. $\dim V_1=2$). Suppose $\xi$ is not injective, thus there exists two 
non-collinear vectors $x$ and $y$ of $V_1$ such that $\xi(x)=\xi(y)$.
Now $(x,y)$ is a basis of $V_1$, so for any $z\in \mathbb{P}(V_1)$, by the defining property of $\xi$,
we have $\xi(z)=\xi(x)=\xi(y)$. So $\xi$ is constant.

Now suppose $\dim V_1\geq 3$. We already know that $\xi$ is
either injective or constant on any projective line.
Assume that overall $\xi$  is neither injective nor constant. This means that there exist two distinct points $ x,y$ such that $ \xi(x)=\xi(y)$, and a third point $ z$ satisfying $ \xi(z)\neq \xi(x)$. This implies that $ x,y,z$ are not (projectively) aligned, so they span a projective plane. 
The reader may now wish to follow the proof on Figure \ref{fig:proj}.
Take a point $ w\notin \{y,z\}$ on the line $(yz)$ spanned by $y$ and $z$. Because $ \xi$ is a bijection on both lines $ (yz)$  and $(xz)$,
and the image of both lines under $\xi$ being the same namely $(\xi(y)\xi(z))$, we can find $ w'\notin  \{x,z\}$ on $ (xz)$ such that $ \xi(w)=\xi(w')\neq \xi(x)$. Now consider the intersection $ u=(ww')\cap (xy)$ in the projective plane $\Span(x,y,z)$. Then  we have $ \xi(u)=\xi(x)=\xi(y)\neq \xi(w)$, so that on the line $ (ww')$ the map $ \xi$ is neither constant nor injective, a contradiction.
\begin{figure}
\psset{xunit=1.0cm,yunit=1.0cm,algebraic=true,dimen=middle,dotstyle=o,dotsize=3pt 0,linewidth=0.8pt,arrowsize=3pt 2,arrowinset=0.25}
\psscalebox{1.3}{
\begin{pspicture*}(-6.02,-0.64)(3.88,5.78)
\psplot{-4.42}{4.88}{(-15.35-4.9*x)/-3.46}
\psplot{-4.42}{4.88}{(--4.11--1.62*x)/6.5}
\psplot{-4.42}{4.88}{(-15--3.28*x)/-3.04}
\psplot{-4.42}{4.88}{(--4.4-0.17*x)/3.93}
\begin{scriptsize}
\psdots[dotstyle=*](0.2,4.72)
\rput[bl](0.11,4.9){$x$}
\psdots[dotstyle=*](-3.26,-0.18)
\rput[bl](-3.38,-0.05){$z$}
\psdots[dotstyle=*](3.24,1.44)
\rput[bl](3.2,1.55){$y$}
\psdots[dotstyle=*](-2.27,1.22)
\rput[bl](-2.5,1.36){$w'$}
\psdots[dotstyle=*](1.66,1.05)
\rput[bl](1.54,1.17){$w$}
\psdots[dotstyle=*](3.69,0.96)
\rput[bl](3.48,0.65){$u$}
\end{scriptsize}
\end{pspicture*}}

\caption{Proof of Lemma \ref{collineation}.}
\label{fig:proj}
\end{figure}
\end{proof}

Finally, we recall the fundamental
 theorem \cite[Th\'eor\`eme 7]{Samuel} of projective geometry.
\begin{theorem}
\label{fundamental}
Suppose that $\xi : \mathbb{P}(V_1)\rightarrow \mathbb{P}(V_2)$ is injective and has the property 
that for any $x,y,z$ in $V_1$ such that $z\in\Span(x,y)$, we have
$\xi(z)\in\Span(\xi(x),\xi(y))$ (i.e. it maps points on a line to points on a line).
Further, suppose that $\dim V_1\geq 3$.
Then $\xi$ is a projective map,
that is, there exists a linear injection $f : V_1\rightarrow V_2$
such that $\xi([x])=[f(x)]$ for any $x\in V_1$.
\end{theorem}
\noindent Here we require 
the field $\F_p$ to be a prime field; on a non prime finite field $\F_q$, we would need to incorporate Frobenius field automorphisms.

Note that the result holds even if $\dim V_1=2$ in the case where 
$p=2$ or $3$. Indeed, the number of bijections between two projective lines is $(p+1)!$. On the other hand, since there are $(p^2-1)(p^2-p)$
linear bijections between any two given planes, the number of projective bijections is $(p^2-1)(p^2-p)/(p-1)=(p+1)p(p-1)$. These two numbers are equal when $p\in \{2,3\}$ which forces any bijection to be projective.

Now we state this section's main result.

\begin{proposition}
\label{projectif}
Let $P\subset V_1\times V_2$ be a transverse set.
Suppose that $\codim_{V_2} P_{x\cdot}\leq 1$ for any $x\in V_1$.
Then one of the three alternatives holds.
\begin{enumerate}
\item There exist a subset $W\subset V_1$ which is empty or a subspace, and a hyperplane $H\leq V_2$, such that
$P=W\times V_2\cup V_1\times H$.
\item There exists a bilinear form $b$ on $V_1\times V_2$ such that
$P=\{(x,y)\in V_1\times V_2: b(x,y)=0\}$.
\item We have $p\geq 5$ and the minimal codimension of a subspace $W\leq V_1$ such that
$W\times V_2\subset P$ is exactly 2.
\end{enumerate}
\end{proposition}
\noindent Observe that this implies Theorem \ref{positive}, since
 the first two alternatives correspond to bilinear sets.
 This is obvious for the second one.
For the first one, if $W$ is empty, it is clear; otherwise, let $a_1, \ldots , a_k$ be linearly independent linear forms such that $W$ is the intersection of their kernels, and $\ell$ be a linear form that defines $H$. Then 
$$
P = \{(x,y)\in V_1\times V_2 : a_1(x)\ell(y) = \cdots = a_k(x)\ell(y) = 0\}.
$$
One can check that one can not write $P$
as in \eqref{eq:bilinar-bogol} with $W_1$ and $W_2$ other than 
$V_1$ and $V_2$ and with $r_3\neq k$, and $k$ may tend to infinity with $\dim V_1$, while the density is bounded below by $1/p$, but this is not
a contradiction with Theorem \ref{thm:bilinear-bogol}, since $P$ \emph{contains} (but may not be equal to) the Cartesian product $V_1\times H$. As for the last alternative,  Theorem \ref{existence} (ii) indicates that it is not necessarily a bilinear set.
\begin{proof}
Without loss of generality 
suppose that 
$P_{0\cdot}=V_2$. Indeed,  otherwise 
$P_{0\cdot}$ is a hyperplane $H$ and
Lemma \ref{rigidityPx} \textit{(i)} shows that
 $P=V_1\times H$. Let $(x,y)\mapsto x\cdot y$ be a bilinear form of full rank on $V_1\times V_2$. For $\phi \in V_2$
let $\phi^{\perp}=\{y\in V_2: x\cdot \phi=0\}$.
The hypothesis allows us to write $P_{x\cdot}=\xi(x)^{\perp}$ for some vector $\xi(x)\in V_2$ that is defined uniquely
up to homothety.
The proof consists in deriving rigidity properties for $\xi$ which will eventually make
it linear or constant.

With this new notation, the assumption just made implies that $\xi(0)=0$. Further, the second point
of Lemma \ref{rigidityPx} means that $\xi(x)$ depends only
on $[x]$ for $x\neq 0$ and the third point of that lemma  yields
that
whenever $[z]$ is on the projective line spanned by $x$ and $y$,
we have $\xi(z)\in\Span(\xi(x),\xi(y))$.
Using Lemma \ref{rigidityPx} $(iii)$, one can see that  the set
\[W:=\{x\in V_1: P_{x\cdot}= V_2\}\]
 is
a vector subspace.
If $W=V_1$, we have $P=V_1\times V_2$ so the first alternative holds. Otherwise $W\neq V_1$.
Let $V'_1=V_1/W$ and observe that for any given $ x-y=w\in W$, we have 
$ \xi(x)\in\Span(\xi(y),\xi(w))=\Span(\xi(y))$, that is, $\xi(x)=\xi(y)$ up to homothety, so that $\xi$ descends to a map $\xi': \mathbb{P}(V_1/W)\rightarrow \mathbb{P}(V_2)$. Thus $\xi'$ is a map $\mathbb{P}(V'_1)\rightarrow \mathbb{P}(V_2)$ that maps aligned points
to aligned points. If $\codim W=1$, it follows that $[\xi(x)]$
is a nonzero constant vector $\phi$ for $x\in V\setminus W$ so the first alternative is true with $H=\phi^{\perp}$. In the following we assume that $\codim W\geq 2$.

By construction $\xi'$ satisfies the hypothesis of Lemma  \ref{collineation}, therefore it should be either constant
or injective. If $\xi'$ is constant on $\mathbb{P}(V'_1)$, we can take $\xi(x)$ to be a nonzero constant vector $\phi\in V_2$ for all $x\in W^{\perp}$, while $\xi(x)=0$ on $W$. Let $H$ denote the subspace orthogonal to $\phi$. Then $P=W\times V_2\cup V_1\times H$,
which is the first alternative. We suppose now that $\xi'$ is injective. If $\dim V'_1=2$ and $p\geq 5$, the third alternative is true. Now suppose
that $\dim V'_1\geq 3$ or that $\dim V'_1=2$ and $p\in\{2,3\}$.
Theorem \ref{fundamental} and the remark following it 
 imply that $\xi$ comes from an injective linear map $V_1'\rightarrow V_2$, which we extend to a linear map 
$f:
V_1\rightarrow V_2$ with kernel $W$. In the particular case $p\in\{2,3\}$ this proves proposition \ref{positive}. 
Then $P$ is the zero set of the bilinear form $(x,y)\mapsto f(x)\cdot y$, which concludes the proof of Proposition \ref{projectif}.
\end{proof}

\section{Proof of proposition \ref{existence}}
\label{thirdAlt}
First we introduce a new notation and a characterisation of bilinear sets.
For a set $P\subset V_1\times V_2$ satisfying
$P_{0\cdot}=V_2$ and $P_{\cdot 0}=V_1$, let
$\Ann(P)$ be the subspace of the space $\mathcal{B}(V_1,V_2)$
of bilinear forms on $V_1\times V_2$ that consist of the forms
that vanish on $P$.
For a set $M\subset \mathcal{B}(V_1,V_2)$, let $\Orth(M)$
be the (bilinear) subset $V_1\times V_2$
where all the forms of $M$ vanish simultaneously.
Thus in general $P\subset \Orth(\Ann(P))$, while the equality
holds if and only if $P$ is a bilinear set.

Now we prove Theorem \ref{existence} (i), that is,
we show that some transverse sets satisfying
the third alternative of Proposition \ref{projectif} are not bilinear. 
Let $W$ be a subspace of codimension 2 in $V_1$.
Let $V'_1=V_1/W$ and $\xi' : \mathbb{P}(V'_1)\rightarrow \mathbb{P}(V_2)$ be a non-projective bijection onto a projective line;
as observed after Theorem \ref{fundamental}, this is possible when $p\geq 5$ since there are $(p+1)!$ bijection between any two projective lines but only 
$(p+1)p(p-1)$ projective maps between them.
Extend naturally $\xi'$ to a map $\xi : V_1\rightarrow V_2$ that induces $\xi'$ by projection
and let $P=\bigcup_{x\in V_1}\{x\}\times \xi (x)^\perp$.
Thanks to the characterization from Lemma \ref{rigidityPx}, we see that $P$ is transverse.

Let $b\in \Ann(P)$, one can write 
$b(x,y)= f(x)\cdot y$ where $f$ is a linear map $V_1\rightarrow V_2$ vanishing on $W$; thus it induces a linear map $f':V'_1\rightarrow V_2$ satisfying $f'(x)\in \Span(\xi'(x))$ for $x\in V'_1\setminus\{0\}$. Recall that $W$ has codimension two and therefore $f'$ has either rank $2$, $1$ or $0$ respectively. In the first case $f'$ does not vanish on $V'_1\setminus\{0\}$ and we get $\xi'(x)=[f'(x)]$ for any $x\neq 0$. As a consequence $\xi'$ is projective,
which is false. The second possibility can be ruled out too. Indeed, in this case the image of $f'$ is a line $\ell$, i.e. a vector space of dimension one.
As a consequence $\xi'([x])$ will have the same constant value for any $x\in V_1'\setminus\ker f'$ which contradicts the fact that it is injective by construction. The only possibility left is $f'=0$. This proves that $\Ann(P)=\{0\}$ and so $\Orth(\Ann(P))=V_1\times V_2\neq P$, which means that $P$ is not bilinear, concluding the proof of Theorem \ref{existence} (i).

We now show Theorem \ref{existence} (ii).
Here we think of $V_1$ and $V_2$ as two $n$-dimensional $\F_p$-vector spaces. 
Recall the characterisation of transverse sets obtained in Lemma \ref{rigidityPx}.
In particular, if $P_{x\cdot}\cap P_{y\cdot}=\{0\}$ for any $[x]\neq [y]$, the third property of that Lemma \ref{rigidityPx} is vacuous. As a consequence the characterization of transverse sets it provides is easier to satisfy.
One can achieve this, for instance, by taking a bijection
 $\sigma: \mathbb{P}(V_1)\rightarrow \mathbb{P}(V_2)$
 and letting
$P$ be the transverse set
\[P_\sigma=\{0\}\times V_2\cup \bigcup_{x\in \mathbb{P}(V_1)}\Span(x)\times\Span(\sigma(x))\]

\noindent where
span denotes the linear span in $V_1$ or $V_2$. 

With the assistance of a computer, it is possible to find $\sigma$ 
such that $P_\sigma\neq \Orth(\Ann(P_{\sigma}))$ for small $p$ and $n$. For instance, for $p=2$
and $n=3$ one can let $\sigma$ be
the permutation of
$\mathbb{P}(\F_2^3)=\F_2^3\setminus \{(0,0,0)\}$  defined in Figure 2.
The above characterization implies that $P_\sigma$ is not a bilinear set. Indeed, we find that

\begin{figure}
\label{sigma}
\begin{tabular}{|c|c|c|c|c|c|c|c|}
\hline 
$x$ & $(1,0,0)$ & $(0,1,0)$ & $(0,0,1)$ & $(1,1,0)$ & $(1,0,1)$ & $(0,1,1)$ & $(1,1,1)$ \\ 
\hline 
$\sigma(x)$ & $(1,0,0)$ & $(0,1,0)$ & $(0,0,1)$ & $(1,1,0)$ & $(0,1,1)$ & $(1,1,1)$ & $(1,0,1)$ \\ 
\hline 
\end{tabular} 
\caption{Table defining the permutation $\sigma$.}
\end{figure}

$$\Ann(P)=\left\lbrace 
\begin{pmatrix}
0 & 0 & 0 \\ 
0 & 0 & 0 \\ 
0 & 0 & 0
\end{pmatrix},
\begin{pmatrix}
0 & 0 & 0 \\ 
0 & 0 & 1 \\ 
1 & 0 & 0
\end{pmatrix},
\begin{pmatrix}
0 & 0 & 1 \\ 
0 & 0 & 1 \\ 
0 & 1 & 0
\end{pmatrix}  ,
\begin{pmatrix}
0 & 0 & 1 \\ 
0 & 0 & 0 \\ 
1 & 1 & 0
\end{pmatrix} 
\right\rbrace
$$
so that $\Orth(\Ann(P))$ contains $((1, 0, 0), (0, 1, 0))$,
an element which does not belong to $P$, so $P$
is not bilinear.

For general $p$ and $n$, the following non-constructive counting argument shows that
there exists a permutation $\sigma$ such that $P_\sigma$ is not bilinear.
On the one hand, the number of points in a projective space can be bounded from below, i.e.
\[\abs{\mathbb{P}(V_1)}=\frac{p^n-1}{p-1}\geq p^{n-1}.\]
Thus there are at least 
\[p^{n-1}!\geq \left(\frac{p^{n-1}}{e}\right)^{p^{n-1}}\]
transverse sets $P_\sigma$, where we used the inequality
$e^m\geq m^m/(m!)$ valid for any positive integer $m$.
On the other hand, the number of subspaces $M$ of 
$\mathcal{B}(V_1,V_2)$
can be bounded from above as follows. The space of bilinear forms $\mathcal{B}(V_1,V_2)$ has dimension $n^2$ and contains $p^{n^2}$ elements. The number of subspaces of dimension $k$ can be bounded by $p^{kn^2}$. Recall that there exists the same number of spaces of dimension $k$ and $n^2-k$ so the total number of subspaces can be bounded above by $\sum_{k=0}^{n^2} |\{H\subset \mathcal{B}(V_1,V_2): \dim(H)=k\}| \leq 2\frac{p^{n^4/2+n^2}-1}{p^{n^2}-1}$ (if $n$ is even this is clear and if it is odd the number of subspaces of dimension $(n^2+1)/2$ is only counted once and the bound obtained is smaller than the one given) . Now we argue by contradiction. The absence of counterexamples would force ($p,n\ge 2$)
\[\left(\frac{p^{n-1}}{e}\right)^{p^{n-1}}\leq p^{n-1}!\leq 2\frac{p^{n^4/2+n^2}-1}{p^{n^2}-1} \leq \frac{32}{15} p^{n^4/2}\]
which provides the contradiction we were seeking for $n\geq n_0(p)$. Indeed, we can take $n_0(p)=11$ for all $p$ but this estimate can be improved if we allow $p$ to be large enough and for instance $n_0(p)=2$ is enough for $p\ge 13$.

\section{Acknowledgments}

The first author thanks Mark Pankov
for a useful e-mail conversation that gave the idea for Theorem \ref{existence} (ii). The third named author is grateful to Carlos Pastor for his careful reading of an early version of Section \ref{galg}.

The second and third authors were partially supported by  MTM2014-56350-P  project of the MCINN (Spain). This material is based upon work supported by the National Science
Foundation under Grant No. DMS-1638352.

\end{document}